\documentclass{amsart}
\usepackage{amssymb}
\usepackage{mathrsfs}
\usepackage{graphicx}
\usepackage{comment}
\usepackage{enumitem}

\newcommand{\ds}{\displaystyle}
\newcommand{\be}{\begin{equation}}
\newcommand{\ee}{\end{equation}}
\newcommand{\ba}{\begin{align}}
\newcommand{\ea}{\end{align}}

\newtheorem{theorem}{Theorem}[section]

\newtheorem{lemma}{Lemma}[section]

{\begin{list}{}{%

\settowidth{\labelwidth}{\textsf{{\it #1.}}}
\setlength{\labelsep}{4mm}%
\setlength{\leftmargin}{\labelwidth}
\addtolength{\leftmargin}{\labelsep}
}}
{\end{list}}

\def\beq{\begin{equation}}\def\enq{\end{equation}}

\keywords{inhomogeneous Diophantine approximation}
\subjclass[2010]{Primary: 11J20; Secondary: 11J06, 11J70}
\date{\today}

\title{Bounding the Largest Inhomogeneous Approximation Constant}

\author[B. Paudel]{Bishnu Paudel}
\address{ Department of Mathematics\\
         Kansas State University\\
         Manhattan, KS 66506, USA}
\email{bpaudel@ksu.edu, pinner@math.ksu.edu}

\author[C. Pinner]{Chris Pinner}

\begin{document}

\begin{abstract} For a given irrational number $\alpha$ and a real number $\gamma$ in $(0,1)$ one defines the two-sided inhomogeneous approximation constant 
\begin{equation*}
    M(\alpha,\gamma):=\liminf_{|n|\rightarrow\infty}|n| ||n\alpha-\gamma||,
\end{equation*}
and the case of worst inhomogeneous approximation for $\alpha$
\begin{equation*}
    \rho(\alpha):=\sup_{\gamma\notin\mathbb{Z}+\alpha\mathbb{Z}}M(\alpha,\gamma).
\end{equation*}
We are interested in  lower bounds on $\rho(\alpha)$ in terms of $R:=\liminf_{i\rightarrow\infty}a_i,$ where the $a_i$ are the partial quotients in the negative (i.e.\ the `round-up') continued fraction expansion  of $\alpha$.
We obtain bounds for any  $R\geq 3$ which are best possible when $R$ is even (and asymptotically precise when $R$ is odd). In particular when $R\geq 3$
$$ \rho(\alpha)\geq \cfrac{1}{6\sqrt{3}+8}=\cfrac{1}{18.3923\dots}, $$
and when $R\geq 4$, optimally,
$$ \rho(\alpha) \geq \cfrac{1}{4\sqrt{3}+2}=\cfrac{1}{8.9282\ldots}. $$

\end{abstract}

\maketitle

\section{Introduction}

For an irrational number $\alpha$ and a real number $\gamma$, we define the two-sided inhomogeneous approximation constant by 

\begin{equation*}
    M(\alpha,\gamma):=\liminf_{|n|\rightarrow\infty}|n|||n\alpha-\gamma||,
\end{equation*}

\noindent where $||x||$ denotes the distance from $x$ to the nearest integer. Plainly this reduces to the classical  homogeneous  problem $\gamma=0$ if $\gamma=m+l\alpha$ for some $m,l\in\mathbb{Z}$. The homogeneous problem is well understood,  with $M(\alpha,0)$ readily determined from the continued fraction expansion of $\alpha=[a_0;a_1,a_2,\ldots],$
$$ M(\alpha,0) =\frac{1}{\limsup_{i\rightarrow \infty} a_i+[0;a_{i+1},a_{i+2},\ldots ]+[0;a_{i-1},a_{i-2},\ldots]}\leq \frac{1}{\sqrt{5}}, $$
leading naturally to bounds in terms of the largest partial quotients 
$$\frac{1}{\sqrt{r^2+4r}}\leq  M(\alpha,0)\leq \frac{1}{\sqrt{r^2+4}}, \quad r:= \limsup_{i\rightarrow \infty} a_i,$$
with equality for $\alpha=[0;\overline{1,r}]=\frac{1}{2}(\sqrt{r^2+4r}-r)$ and $\alpha=[0;\overline{r}]=\frac{1}{2}(\sqrt{r^2+4}-r)$.

For any $\alpha$, we define the worst inhomogeneous approximation 

\begin{equation*}
    \rho(\alpha):= \sup_{\gamma\notin\mathbb{Z}+\alpha\mathbb{Z}}M(\alpha,\gamma).
\end{equation*}
By contrast with the homogeneous case, the inhomogeneous constant $\rho(\alpha)$  will be affected by the smallest partial quotients
\be \label{defR} R:=\liminf_{i\rightarrow \infty} a_i, \ee
\noindent From a well-known theorem of Minkowski we have $$\rho(\alpha)\leq\frac{1}{4},$$ 
see, for example, \cite[Chap. III]{Cassels} or \cite[IV.9]{Rockett},  Grace \cite{grace} giving examples 
with $R=\infty$ and $\rho(\alpha)=\frac{1}{4}$. We are interested here in the lower bound for $\rho(\alpha)$. Absolute bounds 
\be \label{lowerC}
    \rho(\alpha)\geq C
\ee
have some history. Davenport \cite{Davenport} obtained \eqref{lowerC} with $C=\frac{1}{128}$,  Ennola \cite{Ennola} 
 $$C=\frac{1}{16+6\sqrt{6}}=\frac{1}{30.69...},$$
and in \cite{Arxiv} the absolute lower bound was improved to 
\be \label{bestC} C=\frac{(\sqrt{10}-3)(7-\sqrt{13})}{(31-2\sqrt{10}-3\sqrt{13})}=\frac{1}{25.1592...}.\ee
See Rockett and Sz\"{u}sz \cite{Rockett} for a simpler proof with $C=\frac{1}{32}.$
The smallest known value of $\rho(\alpha)$, and hence an upper bound on the optimal absolute lower bound $C$,  is still  an example of Pitman \cite{Pitman} 
\begin{equation*}
    \rho\left(\frac{\sqrt{3122285}-1097}{1094}\right)=\frac{547}{4\sqrt{3122285}}=\frac{1}{12.9213...}.
\end{equation*}  
More generally, \cite{Arxiv} obtains bounds of the form $\rho (\alpha)\geq C^*(R)$, where the $a_i$ in \eqref{defR} are
 the partial quotients in the nearest integer continued fraction of $\alpha$
(giving \eqref{bestC} when $R=2$ and an improvement when $R\geq 3$). The bound comes by constructing a $\gamma^*$ with $M(\alpha,\gamma^*)\geq C^*(R)$.
The values for small $R$ are given in \eqref{table}  below and the asymptotic behavior in \eqref{ArxivAsym}.
The goal here is to improve these  $R\geq 3$ bounds when the $a_i$ in \eqref{defR} are the partial quotients in the negative continued fraction expansion of $\alpha$ rather than the nearest integer expansion.

\section{Preliminaries}

Different algorithms have been used for computing $M(\alpha,\gamma)$, see Komatsu \cite{Takao}. In this paper we will follow the approach of  \cite{Pinner}, which showed how $M(\alpha,\gamma)$ can be expressed in terms of the negative continued fraction expansion of $\alpha$ and a corresponding $\alpha$--expansion of  $\gamma \not\in \mathbb{Z}+\alpha\mathbb{Z}$. We start by recalling some notations and results from \cite{Pinner}. Since $||m+x||=||x||$ for any integer $m$, we may  assume that $\alpha,\gamma\in(0,1)$. For an $\alpha\in(0,1)$ we define the negative continued fraction expansion 
\begin{equation}\label{ContExpan}
    \alpha=\frac{1}{a_1-\cfrac{1}{a_2-\cfrac{1}{a_3-\cdots}}}=:[0; a_1, a_2, a_3,\cdots]^{-},
\end{equation}
where the integers $a_i\geq2$ are generated by the algorithm 
\begin{equation*}
    \alpha_0:=\{\alpha\}=\alpha, \,\ \,\ a_{n+1}:=\left\lceil{\frac{1}{\alpha_n}}\right\rceil, \,\ \,\ \alpha_{n+1}:=\left\lceil{\frac{1}{\alpha_n}}\right\rceil-\frac{1}{\alpha_n},
\end{equation*}
with the corresponding convergents $\cfrac{p_n}{q_n}:=[0;a_1, a_2,\ldots, a_n]^-$ given by 
\begin{align*}
    p_{n+1}&:=a_{n+1}p_n-p_{n-1}, \,\ p_0=0, \,\ p_{-1}=-1,\\
    q_{n+1}&:=a_{n+1}q_n-q_{n-1}, \,\,\ q_0=1, \,\,\ q_{-1}=0.
\end{align*}
We define
\begin{equation*}
    \alpha_n:=[0; a_{n+1}, a_{n+2}, ...]^-, \,\ \,\ \bar\alpha_n:=[0; a_{n}, a_{n-1}, ..., a_1]^-,\,\ \,\ D_n:=q_n\alpha-p_n, 
\end{equation*}
so that $$D_n=\alpha_0\alpha_1\cdots\alpha_n=a_nD_{n-1}-D_{n-2},\,\ \,\ q_n=(\bar\alpha_1\bar\alpha_2\cdots\bar\alpha_n)^{-1}.$$
We observe that 
\begin{align}\label{1}
    (a_1-1)D_0+\sum_{i=2}^{\infty}(a_i-2)D_{i-1}=1,\,\ \,\ p_{n+1}q_n-p_n q_{n+1}=1.
\end{align}

For any real number $\gamma\in(0,1)$, we generate the integers $b_i$ by the algorithm 
\begin{equation*}
    \gamma_0:=\{\gamma\}=\gamma, \,\ \,\ b_{i+1}:=\left\lfloor\frac{\gamma_i}{\alpha_i}\right\rfloor, \,\ \,\ \gamma_{i+1}:=\left\{\frac{\gamma_i}{\alpha_i}\right\},
\end{equation*}
so that 
\begin{equation*}
    \gamma=\sum_{i=1}^{n}b_iD_{i-1}+\gamma_nD_{n-1}=\sum_{i=1}^{\infty}b_iD_{i-1}
\end{equation*}
gives the unique expansion of $\gamma$ of the form $\sum_{i=1}^{\infty}b_iD_{i-1}$, called the $\alpha$--$expansion$ of $\gamma$, with the following properties \cite{Pinner}:
\begin{enumerate}
    \item $0\leq b_i\leq a_i-1$ for all $i$,
    \item the sequence $\{b_i\}_i$ does not contain a block of the form $b_s=a_s-1$ for some $s$, with $b_j=a_j-2$ for all $j>s$ or with $b_k=a_k-1$ for some $k>s$ and $b_j=a_j-2$ for all $k>j>s$.
\end{enumerate}
We define the sequence of integers $t_k$ by $b_k=\frac{1}{2}(a_k-2+t_k)$
\begin{equation}\label{generalgamma}
    \gamma=\sum_{i=1}^{\infty}\frac{1}{2}(a_i-2+t_i)D_{i-1}.
\end{equation}
and 
\begin{align*}
    d_k^-&:=\sum_{j=1}^{k}t_j\left(\frac{q_{j-1}}{q_k}\right)=t_k\bar\alpha_k+t_{k-1}\bar\alpha_k\bar\alpha_{k-1}+t_{k-2}\bar\alpha_k\bar\alpha_{k-1}\bar\alpha_{k-2}+\cdots ,\\
    d_{k}^{+}&:=\sum_{j=k+1}^{\infty}t_j\left(\frac{D_{j-1}}{D_{k-1}}\right)=t_{k+1}\alpha_{k}+t_{k+2}\alpha_{k}\alpha_{k+1}+t_{k+3}\alpha_{k}\alpha_{k+1}\alpha_{k+2}+\cdots .
\end{align*}
Notice that $t_k$ and $a_k$ have the same parity, and $-(a_k-2)\leq t_k\leq a_k$. It was observed in \cite{Pinner} that 
\begin{equation}\label{drange}
    -(1-\bar{\alpha}_k)\leq d_k^-\leq(1+\bar{\alpha}_k), \,\ \,\ -(1-\alpha_k)\leq d_k^+\leq(1+\alpha_k),
\end{equation}
with $d_k^+\geq1-\alpha_k$ (respectively $d_k^-\geq1-\bar{\alpha}_n$) if and only if the sequence $t_{k+1}, t_{k+2},...$ (respectively $t_{k}, t_{k-1},...$) has the form $t_j=a_j$ for some $j>k$ (respectively $j\leq k$) with $t_i=a_i-2$ for any $k<i<j$ (respectively $j<i\leq k$). Note that $t_i=a_i$ if and only if $b_i=a_i-1$.
When only finitely many of the $b_i=a_i-1$, it was shown in \cite[Theorem 1]{Pinner} that  the sequence of best positive and negative  inhomogeneous  approximations lies  amongst the
$$Q_k:=\sum_{i=1}^kb_iq_{i-1},\;\;\; Q_k+q_{k-1}, \;\;\; -(q_k-Q_k),\;\; \; -(q_k-q_{k-1}-Q_k). $$
\begin{lemma}\label{mainlemma}
If $\gamma\notin\mathbb{Z}+\alpha\mathbb{Z}$ and the $\alpha$--$expansion$ of $\gamma$ has $t_i=a_i$ at most finitely many times, then 
\begin{equation*}
    M(\alpha,\gamma)=\liminf_{k\rightarrow\infty}\min\{s_1(k), s_2(k), s_3(k), s_4(k)\},
\end{equation*}
where 
\begin{align*}
    s_1(k)&:=\frac{1}{4}(1-\bar\alpha_k+d_k^-)(1-\alpha_k+d_k^+)/(1-\bar\alpha_k\alpha_k),\\
     s_2(k)&:=\frac{1}{4}(1+\bar\alpha_k+d_k^-)(1+\alpha_k-d_k^+)/(1-\bar\alpha_k\alpha_k),\\
      s_3(k)&:=\frac{1}{4}(1-\bar\alpha_k-d_k^-)(1-\alpha_k-d_k^+)/(1-\bar\alpha_k\alpha_k),\\
       s_4(k)&:=\frac{1}{4}(1+\bar\alpha_k-d_k^-)(1+\alpha_k+d_k^+)/(1-\bar\alpha_k\alpha_k).
\end{align*}
\end{lemma}
 We set $R:=\liminf_{i\rightarrow\infty}a_i$, where the $a_i$ are now  the partial quotients in the negative expansion \eqref{ContExpan}. When $R\geq 3$, an upper bound for $\rho(\alpha)$ was given in \cite[Corollary 1]{Pinner}  
\begin{equation}\label{upperbound}
    \rho(\alpha)\leq\frac{1}{4}\left(1-\frac{1}{R}\right).
\end{equation}
This is best possible when  $R$ is even  with $\rho\left([0;\overline{R,2k}]^-\right)\rightarrow \frac{1}{4}(1-1/R)$ as $k\rightarrow \infty$.
Our goal here is  to obtain a lower bound for $\rho(\alpha)$ when $R\geq 3$. For this, we first construct a $\gamma^*\in(0,1)$ and then we use Lemma \ref{mainlemma} to compute $M(\alpha,\gamma^*)$, which gives a lower bound $\rho(\alpha)\geq M(\alpha,\gamma^*)$. \\

\section{Main results}

Consider a real number $\gamma^*\in(0,1)$ which has the unique $\alpha$--expansion
\begin{equation}\label{gamma*}
    \gamma^*=\sum_{i=1}^{\infty}b_iD_{i-1}=\sum_{i=1}^{\infty}\frac{1}{2}(a_i-2+t_i)D_{i-1},
\end{equation}
where the sequence $\{t_i\}$ is given by 
 \[
  t_i=
  \begin{cases}
     0, &\text{if $a_i$ is even,}\\
     (-1)^{j+1}, & \text{if $a_i$ is the \textit{j}th odd partial quotient}. 
  \end{cases}
  \]
Notice that any two nonzero consecutive $t_i$ have opposite signs and hence $|d_{k}^{-}|\leq\bar{\alpha}_k,$   $|d_{k}^{+}|\leq\alpha_k,$
 and $d_{k}^{-} d_k^{+}\leq 0$. We define two numbers $\beta$ and $\delta$
 \begin{equation}\label{bd}
    \beta:=[0; \overline{R_*} ]^-=\frac{1}{2}\left(R_*-\sqrt{R_*^2-4}\right), \,\ \,\ \,\ \,\ \delta:=[0; R_{**}, \overline{R_*} ]^-=\frac{1}{R_{**}-\beta},
\end{equation}
where
\[
  R_*,\,\ R_{**}:=
  \begin{cases}
     R,\,\ R+1, &\text{if $R$ is even},\\
     R+1,\,\ R, & \text{if $R$ is odd}.
  \end{cases}
  \]
We set 
\begin{equation}\label{C}
    C(R):=\frac{(1-2\delta)(1-\beta)}{4(1-\delta\beta)},
\end{equation}
 observing that if $R$ is even
$$ C(R)= \frac{1}{4} \left(\frac{R-2}{\sqrt{R^2-4}+1}\right), $$
and if $R$ is odd
$$ C(R)=  \frac{1}{4} \left( \frac{2R-2-\sqrt{(R+1)^2-4}}{\sqrt{(R+1)^2-4}-1}\right).  $$
The value of $M(\alpha,\gamma^*)$ gives us a lower bound for $\rho(\alpha)$.

\begin{theorem}\label{mymaintheorem}
Suppose that (\ref{ContExpan}) gives the negative continued fraction expansion of $\alpha$ and  $R=\liminf_{i\rightarrow \infty}a_i\geq 3$. Then, with $\gamma^*$ as in $(\ref{gamma*})$ and $C(R)$ as in $(\ref{C})$ we have
\be  \label{C>}
  \rho(\alpha)\geq   M(\alpha,\gamma^*)\geq C(R).
\ee
In particular, when $R=3$, 
$$ \rho(\alpha)\geq C(3)=\frac{1}{6\sqrt{3}+8}=\frac{1}{18.3923\ldots},$$
and when $R\geq 4$
$$ \rho(\alpha)\geq C(4)=\frac{1}{4\sqrt{3}+2}=\frac{1}{8.9282\ldots}. $$

\end{theorem}
\noindent
For $R\geq 3$ the value of $C(R)$ improves the lower bound $C^*(R)$ of \cite[Theorem 4]{Arxiv}:

\be
\begin{array}{c|cc} \label{table}
R & C^*(R)^{-1} & C(R)^{-1} \\ \hline
2 & 25.1592\ldots  &  - \\
 3 & 20.4874\ldots  & \ds 18.3923\ldots\\
4 & 9.3372\ldots  &\ds  8.9282\ldots \\
5 & 8.2500\ldots  & \ds 7.9497\ldots\\
6 & 6.8120\ldots  &\ds  6.6568\ldots \\
7 &  6.4643\ldots  & \ds 6.3431\ldots\\ 
8 & 5.9109\ldots  &\ds  5.8306\ldots\end{array} 
\ee

\vskip0.1in
Of course,  if all the $a_i\geq 3$ in the negative   expansion the negative and nearest integer continued fraction expansions coincide. That is,
 the lower bound $C^*(R)$, $R\geq 3$, actually applies to a much larger class of $\alpha$ than $C(R)$ (and so is not surprisingly smaller).
Better bounds are also given in \cite{Arxiv} when the nearest integer expansion coincides with the regular expansion.
\noindent
Notice that the bound  $C(R)$ increases to $1/4$ as $R\rightarrow \infty$; in particular, from (\ref{upperbound}) and Theorem \ref{mymaintheorem},  when $R\geq 3$
\be \label{iff} \rho(\alpha)= \frac{1}{4} \text{ if and only if  } R=\infty. \ee
Fukasawa \cite{Fuk} showed that  \eqref{iff} holds without the $R\geq 3$ condition when  using the nearest integer continued fraction expansion 
(the restriction needed here since large partial quotients in the regular expansion will cause long strings of 2's in the negative expansion).
We note the asymptotic behavior of $C(R)$; when $R\geq 4$ is even
\be \label{AsympEven} C(R) = \cfrac{1}{4}\left(1-\cfrac{3}{R}+\cfrac{5}{R^2}-\cfrac{E_1(R)}{R^3}\right),  \;\; 7.3268 < E_1(R) <11,
\ee
and when $R\geq 3$ is odd
\be  \label{AsympOdd}  C(R)=  \cfrac{1}{4}\left(1-\cfrac{3}{R}+\cfrac{4}{R^2}-\cfrac{E_2(R)}{R^3}\right),\;\;\;  6.1279 < E_2(R)<10.   \ee
For comparison, we note the \cite{Arxiv} bounds
\be \label{ArxivAsym}  C^*(R)=\begin{cases} \frac{1}{4} \left( 1-\frac{3}{R} +\frac{4}{R^2} + O\left(R^{-3}\right)\right), & \text{ if $R$ is even,}\\
\frac{1}{4} \left( 1-\frac{3}{R} +\frac{3}{R^2} + O\left(R^{-3}\right)\right), & \text{ if $R$ is odd,}\end{cases} \ee
with this lower bound asymptotically optimal (and hence $C^*(R)$ inevitably smaller than $C(R)$) when $R$ is even. The optimal  $O(R^{-2})$ term  in a $C^*(R)$ bound remains undetermined when $R$ is odd.

Our lower bound $C(R)$ for $\rho(\alpha)$ is optimal when $R$ is even. 

\begin{theorem}\label{myexample}
With even $R\geq4$, if $\alpha$ has negative continued fraction expansion of period 
$
     R+1, (R,)^l,
$
then 
$\rho(\alpha)\rightarrow C(R)$ as   $ l\rightarrow\infty.$

\end{theorem}

When $R\geq3$ is odd, it will be clear from the proof of Theorem \ref{mymaintheorem} that if $\alpha$ has negative continued fraction expansion of period 
$  R, (R+1,)^l,$
then $M(\alpha,\gamma^*)\rightarrow C(R)$ as $l\rightarrow\infty.$ So the bound $M(\alpha,\gamma*)\geq C(R)$ in Theorem \ref{mymaintheorem} is  still best possible.  However $\gamma^*$ is no longer the best choice of $\gamma$; as we observe at the end of the paper, for $R\geq 5$  these $\alpha$  have
\be \label{otheralpha} \lim_{l\rightarrow \infty} \rho(\alpha)=\frac{\left(1-2\delta+\frac{2\delta\beta}{1+\beta}\right)(1-\beta) }{4(1-\delta\beta)}=\frac{1}{4}\left( 1-\frac{3}{R}+\frac{6}{R^2} + O\left(R^{-3}\right)\right).\ee

We need a more complicated example to show the asymptotic sharpness of  our lower bound when $R$ is odd.

\begin{theorem}\label{asymptotic}
If $R$ is odd and  $\alpha$ has negative continued fraction  expansion
\begin{align*}
    \alpha=[0;\overline{R,R,R+1,R,R+1,R+1,R,R+1,R+1,R,R+1}]^-,
\end{align*}
then
\begin{equation*}
    \rho(\alpha)=\frac{1}{4}\left(1-\frac{3}{R}+\frac{4}{R^2}+O\left(R^{-3}\right)\right).
\end{equation*}
\end{theorem}
For $R=3,5$ and 7 the period two examples
\begin{align*}
\rho\left([0;\overline{3,5}]^-\right) & =\frac{13}{11\sqrt{165}}=\frac{1}{10.8690\ldots}, \\
 \rho\left([0;\overline{5,6}]^-\right) &=\frac{589}{312\sqrt{195}}=\frac{1}{7.3970\ldots},\\
 \rho\left([0;\overline{7,8}]^-\right) &=\frac{3649}{1664\sqrt{182}}=\frac{1}{6.1519\ldots}, 
\end{align*}
 from \cite{Pinner2}
give upper bounds on the optimal $C(R)$.

\section{Proof Of Theorem 1}

We shall make frequent use of the following simple observation.

\begin{lemma} \label{calculus} If $\lambda >\mu>0$ then  $\displaystyle f(z)=\frac{1-\lambda z}{1-\mu z}$ is decreasing for $0\leq \lambda z <1$,
\end{lemma}

In particular, if $\lambda_1,\lambda_2>1$ and $0\leq x\leq \alpha,$ $0\leq y\leq \beta$,  with $\lambda_1\alpha,\lambda_2\beta<1$, then
$$ \frac{(1-\lambda_1 x)(1-\lambda_2y)}{1-xy}\geq  \left( \frac{1-\lambda_1 \alpha}{1-\alpha y}\right)(1-\lambda_2y) \geq \frac{(1-\lambda_1 \alpha)(1-\lambda_2 \beta)}{1-\alpha \beta}.$$

\begin{proof} Plainly $f'(z)=-(\lambda-\mu)/(1-z\mu)^2<0$ for $0\leq z <\mu^{-1}$.
\end{proof}

\begin{proof}[Proof of Theorem \ref{mymaintheorem}]
From our construction of $\gamma^*$, we have $|d_k^-|\leq\bar\alpha_k$, $|d_k^+|\leq\alpha_k$, and $s_2(k),s_4(k)\geq \frac{1}{4}$.  Hence, by Lemma \ref{mainlemma}, we have 

\begin{equation}\label{s1,s3}
    M(\alpha,\gamma^*)=\liminf_{k\rightarrow\infty}\min\{s_{1}(k), s_{3}(k)\}.
\end{equation}
Since we are evaluating $\liminf$ on $k$, from now on whenever we see the index $k$, it will be understood that we are letting $k\rightarrow\infty$. Also, we may assume that $a_i\geq R$ for all $i$.

Observe that changing the signs of $t_i$ only interchanges $s_1(i)$ with $s_3(i)$. Hence, as long as we check both signs on the $t_i$, it will be enough to show that
$$s_3(k)\geq C(R). $$

We also observe that interchanging the pairs $(a_{k-i}, t_{k-i})$ with $(a_{k+1+i}, t_{k+1+i})$ for all $i\geq0$ only interchanges $\bar\alpha_k$ with $\alpha_k$ and $d_k^-$ with $d_k^+$.

The proof when $R$ is even is straightforward.

\vskip0.1in
\noindent\textbf{Case I: $R$ is even.} In this case we have $\beta =\delta +\delta\beta >\delta$, where
\begin{equation*}
    \beta=[0;\overline{R}]^-\,\ \text{and}\,\ \delta=\frac{1}{R+1-\beta}=[0;R+1,\overline{R}]^-.
\end{equation*}

If $a_k$ is odd and $t_k=1$, then $d_k^-\leq \bar\alpha_k$, $d_k^+\leq 0$, and 
\begin{equation*}
    s_3(k)\geq\frac{(1-2\bar\alpha_k)(1-\alpha_k)}{4(1-\bar\alpha_k\alpha_k)}\geq \frac{(1-2\delta)(1-\beta)}{4(1-\delta\beta)},
\end{equation*}
where the last inequality follows from the Lemma \ref{calculus}, since $\bar\alpha_k\leq \delta, \alpha_k\leq \beta$. 
As observed above this also covers the case $a_{k+1}$ odd with $t_{k+1}=1$. 

If $a_{k}$ is odd and $t_k=-1$ with  $a_{k+1}$ even (likewise $a_{k+1}$ odd, $t_{k+1}=-1$ with $a_k$ even)
we have  $d_k^-\leq 0$, $d_k^+\leq\alpha_k\alpha_{k+1}\leq \alpha_k\beta$ and  Lemma \ref{calculus} with $\alpha_k,\bar\alpha_k\leq \beta$ gives
\be \label{leftover}   s_3(k)\geq \frac{(1-\bar\alpha_k)(1-(1+\beta)\alpha_k)}{4(1-\bar\alpha_k\alpha_k)}\geq\frac{(1-\beta)(1-\beta-\beta^2)}{4(1-\beta^2)}\geq \frac{(1-\beta)(1-2\delta)}{4(1-\delta\beta)},
\ee
since $\delta<\beta$ and $\beta+\beta^2<2\delta$ (equivalently $R\geq 2+2\beta$).

This just leaves the  case  that  $a_k$ and $a_{k+1}$ both are even. If $d_k^-\leq 0$ and $d_k^+\geq 0$ 
(likewise $d_k^-\geq 0$ and $d_k^+\leq 0$) then $d_k^+\leq\alpha_k\alpha_{k+1}\leq \alpha_k \beta$ and again we have \eqref{leftover}.

\vskip0.1in
\noindent\textbf{Case II: R is odd.} In this case we have $\delta=\beta +\delta \beta> \beta$, where
\begin{equation*}
    \beta=[0; \overline{R+1}]^- \,\ \text{and}\,\ \delta=\frac{1}{R-\beta}=[0;R, \overline{R+1}]^-. 
\end{equation*}
We first establish some lemmas. 
Assume in both that $R\geq 3$ is odd and $\gamma=\gamma^*$.

\begin{lemma}\label{oddlemma1}
Suppose that  $\theta<1$. If $a_{k+1}$ is odd and $t_{k+1}=1$, then
\be \label{A>}
    \frac{1-\alpha_k-d_k^+}{1-\theta\alpha_k}\geq\frac{1-2\delta}{1-\theta\delta}.
\ee
Likewise, if $a_k$ is odd and $t_k=1$, then
\be \label{A><}
    \frac{1-\bar\alpha_k-d_k^-}{1-\theta\bar\alpha_k}\geq\frac{1-2\delta}{1-\theta\delta}.
\ee
\end{lemma}

\begin{proof}
Notice that it suffices to show the inequality $(\ref{A>})$ when $k=0$. That is 
\begin{align*}
    A:=\frac{1-\alpha-d_0^+}{1-\theta\alpha}\geq\frac{1-2\delta}{1-\theta\delta},
\end{align*}
where $\alpha_0=\alpha=[0;a_1,a_2,\cdots]^-$, and $d_0^+=t_1\alpha+t_2\alpha\alpha_1+\cdots$.

If  $\alpha\leq \delta $ (for example  the case when the $a_i$, $i\geq 2,$ are all even), then  
\be \label{easy}
    A\geq\frac{1-2\alpha}{1-\theta\alpha}\geq\frac{1-2\delta}{1-\theta\delta},
\ee
from Lemma \ref{calculus}.

So suppose that  $\alpha>\delta$ and  let $a_{n+1}$, $n\geq1$, be the odd partial quotient such that  $a_i$ is even for all $1<i<n+1$. Notice we must have $a_1=a_{n+1}=R$ and $a_i=R+1$ for $1<i<n+1$, else $\alpha <\delta$. Since $t_1=1$ and $t_{n+1}=-1$,
\begin{align*}
    d_0^+\leq \alpha-\alpha\alpha_{1}\cdots\alpha_{n}+\alpha\alpha_{1}\cdots\alpha_{n}\alpha_{n+1}\leq\alpha-\frac{1}{2}\alpha\alpha_{1}\cdots\alpha_{n},
\end{align*}
and
\begin{equation}\label{A}
    A\geq\frac{1-2\alpha+\frac{1}{2}\alpha\alpha_1\cdots\alpha_n}{1-\theta\alpha} >\frac{1-2\alpha+\frac{1}{2}\alpha\alpha_1\cdots\alpha_n}{1-\theta\delta}.
\end{equation}
Setting  $\nu:=[0;a_1, a_2,\ldots,a_n,a_{n+1}+2,a_{n+2},\ldots]^-$ we have $\nu<\delta, $
and  we just need to show that
\be \label{key}  \alpha -\nu \leq  \frac{1}{4} \alpha\alpha_1\cdots\alpha_n, \ee
to obtain $1-2\alpha+\frac{1}{2}\alpha\alpha_1\cdots\alpha_n\geq 1-2\nu\geq 1-2\delta$ and  \eqref{A>}.

Recall that 
\begin{equation}\label{alpha}
\alpha=\frac{p_{n+1}-p_{n}\alpha_{n+1}}{q_{n+1}-q_{n}\alpha_{n+1}}=\frac{p_{n}-p_{n-1}\alpha_{n}}{q_{n}-q_{n-1}\alpha_{n}}.
\end{equation}
Similarly
$$    \nu=\frac{p_{n+1}-p_n\alpha_{n+1}+2p_n}{q_{n+1}-q_n\alpha_{n+1}+2q_n}, $$
and $p_{n+1}q_n-p_n q_{n+1}=1$ gives 
$$\alpha -\nu = \frac{2}{(q_{n+1}-q_{n}\alpha_{n+1})(q_{n+1}-q_n\alpha_{n+1}+2q_n)}= \frac{2\alpha\alpha_1\cdots\alpha_n}{(q_{n+1}+(2-\alpha_{n+1})q_n)},$$
and \eqref{key} just needs $q_{n+1}+(2-\alpha_{n+1})q_n \geq 8$. Plainly $q_{n+1}\geq 3\cdot 3 -1=8.$
 \end{proof}

 \begin{lemma}\label{oddlemma}
 Suppose that $\theta<1$. If  $a_{k+1}$ is even and  $d_{k}^+\leq 0$, then 
 \begin{equation*}
    \frac{1-\alpha_k-d_k^+}{1-\theta\alpha_k}\geq\frac{1-\beta}{1-\theta\beta}.
 \end{equation*}
 \end{lemma}
 \begin{proof}
 We proceed as in the proof of Lemma \ref{oddlemma1}. Suppose $k=0$. Then we show  
 $$A:=\frac{1-\alpha-d_0^+}{1-\theta\alpha}\geq\frac{1-\beta}{1-\theta\beta}.$$
If $\alpha \leq \beta$ then
\be \label{key2}
    A\geq\frac{1-\alpha}{1-\theta\alpha}\geq\frac{1-\beta}{1-\theta\beta}.
\ee
Assume $\alpha >\beta,$ and  let $a_{n+1}, n\geq 1$, be the odd partial quotient such that $a_i$ is even for all $1\leq i\leq n$. Then, since  $t_{1}, t_2,\ldots, t_n=0$ and $t_{n+1}=-1$,
 \begin{align*}
    d_0^+\leq-\alpha\alpha_{1}\cdots\alpha_{n}+\alpha\alpha_{1}\cdots\alpha_{n}\alpha_{n+1}\leq-\frac{1}{2}\alpha\alpha_{1}\cdots\alpha_{n},
\end{align*}
and 
\begin{align*}
    A\geq\frac{1-\alpha+\frac{1}{2}\alpha\alpha_1\cdots\alpha_n}{1-\theta\beta}.
\end{align*}
Set $\nu:=[0;a_1,a_2,\ldots,a_n,a_{n+1}+2,a_{n+2},\ldots]^-<\beta$. This time we just  need to show $\alpha -\nu \leq \frac{1}{2} \alpha\alpha_{1}\cdots\alpha_{n}$, which reduces to $q_{n+1}+(2-\alpha_{n+1})q_n \geq 4$. Plainly  $q_{n+1}\geq 3\cdot 4 -1=11$.
 \end{proof}

\begin{proof}[Proof of Theorem \ref{mymaintheorem} when R is odd] 
We set $\sigma:=[0;\overline{R}]^-$. 

We need to show that $s_3(k)\geq C(R)$. If $a_k$ and $a_{k+1}$ both are odd, then without loss of generality we can assume $t_k=-1$ and $t_{k+1}=1$. Plainly $d_k^-\leq-\bar\alpha_k+\bar \alpha_k\bar\alpha_{k-1}\leq -\bar\alpha_k+\bar\alpha_k\sigma$, and by Lemma \ref{oddlemma1} and Lemma \ref{calculus} (using $\sigma>\delta$ and $\bar\alpha_k\leq \sigma$) and $\sigma>\beta$
$$  s_3(k)\geq\frac{(1-2\delta)(1-\bar\alpha_k \sigma)}{4(1-\bar\alpha_k\delta)}\geq\frac{(1-2\delta)(1- \sigma^2)}{4(1-\sigma \delta)}> \frac{(1-2\delta)(1- \sigma^2)}{4(1-\beta \delta)}>C(R),$$
since $\sigma^2 <\frac{1}{2}\sigma < \beta$.

Now it suffices to consider the following three cases: (i). $a_k$ and $a_{k+1}$ both are even, (ii). $(a_k, t_k)=(odd, -1)$ and $a_{k+1}$ is even, (iii), $(a_k, t_k)=(odd, 1)$ and $a_{k+1}$ is even.
For (i) and (ii), it can be readily seen that
\begin{equation*}
    s_3(k)\geq\frac{1}{4}(1-\sigma)(1-\sigma-\sigma^2)=\frac{1}{4}(1-2\sigma+\sigma^3)\geq \frac{1}{4}(1-2\delta)\geq C(R),
\end{equation*}
using that $\sigma-\delta = (\sigma-\beta)\sigma \delta =(1-\beta+\sigma) \sigma^2\delta\beta < \frac{3}{2}\beta \sigma^3 <\frac{1}{2}\sigma^3.$
For (iii), we apply Lemmas \ref{oddlemma1} and \ref{oddlemma}
\begin{align*}
    s_3(k)&=\frac{(1-\bar\alpha_k-d_k^-)(1-\alpha_k-d_k^+)}{4(1-\bar\alpha_k\alpha_k)}
    \geq  \frac{(1-2\delta)(1-\alpha_k-d_k^+)}{4(1-\delta\alpha_k)}
    \geq \frac{(1-2\delta)(1-\beta)}{4(1-\delta\beta)}. \qedhere
\end{align*}
\end{proof}

It remains just to demonstrate the asymptotics \eqref{AsympEven} and \eqref{AsympOdd}.
For $R\geq 4$ even, one can use $R\beta=1+\beta^2$ to write
$$ E_1(R) = \frac{ 11-2\beta(6-3\beta+\beta^2)}{1+(1-2\beta)/R}=11+O\left(\frac{1}{R}\right), $$
with $E_1(4)=(524-256\sqrt{3})/11=7.3268\ldots$, $E_1(R)\nearrow 11,$ and \eqref{AsympEven} is clear.

For $R\geq 3$ odd, one can use $R\beta=1-\beta+\beta^2$ to write
$$ E_2(R) = \frac{ 10-2\beta(11-7\beta+2\beta^2)}{1-2\beta/R}=10+O\left( \frac{1}{R}\right), $$
with $E_2(3)=(348-162\sqrt{3})/11=6.1279\ldots$, $E_2(R)\nearrow 10,$ and \eqref{AsympOdd} is clear.
\end{proof}

\section{Proof Of Theorem \ref{myexample}}
We assume that  $\alpha$ has expansion  \eqref{ContExpan} of period $R+1, (R,)^l$ with $R\geq 4$ even.

Suppose first that $\gamma$ has an expansion \eqref{generalgamma} with $t_i=0$ when $a_i=R$  and $t_i=\pm 1$
when $a_i=R+1,$ for all sufficiently large $i$. If $a_k=R+1$ and $t_k=1$ then 
$$\bar\alpha_k\rightarrow\delta,\; \alpha_k\rightarrow \beta,\quad  d_k^-\rightarrow\delta,\,\ d_k^+\rightarrow0, \quad 
    s_3(k)\rightarrow C(R),\quad \text{ as } k,l\rightarrow\infty. $$
Likewise, if  $a_k=R+1$ and $t_k=-1,$ then   $s_1(k)\rightarrow C(R)$ as $k,l\rightarrow \infty$.  Hence these $\gamma$ cannot contribute a value $M(\alpha,\gamma)$ strictly greater than $C(R)$ to $\rho(\alpha)$ as $l\rightarrow \infty$.
 By Theorem \ref{mymaintheorem} we have $M(\alpha,\gamma^*)\geq C(R)$ and hence
 \begin{align}\label{M(a,g*)}
    \lim_{l\rightarrow \infty}  M(\alpha,\gamma^*) =C (R).
 \end{align}

It remains to show that $ M(\alpha, \gamma)\leq C(R)$  as $l\rightarrow \infty$ for the remaining $\gamma\notin\mathbb{Z}+\alpha\mathbb{Z};$  that is, those $\gamma$ having an expansion \eqref{generalgamma} with  $|t_i|\geq 2$ infinitely often.

Observe that changing the signs of $t_i$ only interchanges $s_1(i)$ with $s_3(i)$ and $s_2(i)$ with $s_4(i)$. Thus, if we eliminate any block of $t_i$ from consideration, then the same will be true for its negative. Also, interchanging  $\bar\alpha_k$ with $\alpha_k$ and $d_k^-$ with $d_k^+$ does not change $s_1(k)$ and $s_3(k)$ (and interchanges $s_2(k),s_4(k)$). Hence, if we eliminate a block  of $t_{i}$ from consideration, then the same will be true for the reversed block of $t_i$ (on a reversed block of $a_i$).

If $t_k=a_k$ infinitely often, then from Lemma 1 of \cite{Pinner}
\begin{align}\label{s3s4}
   M(\alpha,\gamma)\leq \liminf_{\stackrel{k\rightarrow \infty}{t_k=a_k}}\frac{\bar\alpha_k}{4(1-\bar\alpha_k\alpha_k)},
 \end{align}
 and hence
 \begin{align*}
   M(\alpha,\gamma)\leq\frac{\beta}{4(1-\beta^2)}\leq\frac{(\beta+\beta^2-\delta\beta)}{4(1-\beta^2+\beta^2-\delta\beta)}< C(R),
\end{align*}
on observing that  $(1-2\delta)(1-\beta)=1-3\beta+4\delta\beta$ and $4\beta+\beta^2-5\delta \beta <1$ (equivalently $R>4-\delta(3-2\beta)$).

Thus, we may assume that $t_i= a_i$ at most  finitely many times. In this case, we notice that 
\begin{align*}
    \sqrt{s_3(k)s_4(k)}=\frac{\sqrt{\big((1-d_k^-)^2-\bar\alpha_k^2\big)\big(1-(\alpha_k+d_k^+)^2\big)}}{4(1-\bar\alpha_k\alpha_k)},
\end{align*}
and hence 
\begin{align}\label{1s3s4}
    \min\{s_3(k),s_4(k)\}  \leq\frac{1-d_k^-}{4(1-\bar\alpha_k\alpha_k)}.
\end{align}

We establish the following lemmas.
\begin{lemma}\label{consecutive} 
If the sequence $\{t_i\}_i$ in the expansion (\ref{generalgamma}) of $\gamma$ has infinitely many  blocks of the form  $t_{k}, t_{k+j}>0$, for some $j>0$,  with at least one of $a_k$, $a_{k+j}$ even, and $t_i=0$ for any $k<i<k+j$,  then $M(\alpha, \gamma)< C(R)$ as $l\rightarrow \infty$.
\end{lemma}
\begin{proof}
Without loss of generality suppose that $a_{k}$ is even. Then $a_{k}=R$, $t_k\geq2$, and with $\theta=[0;\overline{R+1}]^-$ we have 
\begin{align*}
    \bar\alpha_k\geq\frac{1}{R-\theta}=\frac{\theta}{1-\theta}, \quad \quad \alpha_k\geq \delta \text{ as } l\rightarrow \infty.
\end{align*}
Plainly  $d_{k}^+\geq0$, while $d_k^-\geq 2\bar\alpha_k + \bar\alpha_k d_{k-1}^->\bar\alpha_k$ by \eqref{drange}, and we get
\begin{align*}
    s_3(k)\leq\frac{(1-2\bar\alpha_k)(1-\alpha_k)}{4(1-\bar\alpha_k\alpha_k)}
    \leq \frac{\left(1-\cfrac{2\theta}{1-\theta}\right)(1-\delta)}{4(1-\delta\beta)}
    \leq\frac{1-3\theta}{4(1-\delta\beta)} <C(R),
\end{align*}
the second inequality from Lemma \ref{calculus} and $\theta/(1-\theta)<\beta$, the third from $\delta>\theta$, and the last since $1-3\theta=1-3\delta+3\delta \theta (\beta-\theta)$ while
$(1-2\delta)(1-\beta)=1-3\delta+\delta\beta$.
\end{proof}
We can now assume that the sequence $\{t_i\}_i$ in the expansion (\ref{generalgamma}) eventually does not contain any block (or its negative) of the type excluded by Lemma $\ref{consecutive}$.
\begin{lemma}\label{>2}
If $\gamma$ has infinitely many $t_k\geq3$, then $M(\alpha,\gamma)<C(R)$.
\end{lemma}
\begin{proof}
If $a_k=R$ and $t_k\geq 4$, then $d_k^-\geq 4\bar\alpha_k+\bar\alpha_k d_{k-1}^-\geq3\bar\alpha_k$ and by $(\ref{1s3s4})$
\begin{align*}
  M(\alpha,\gamma)  \leq\frac{1-3\bar\alpha_k}{4(1-\bar\alpha_k\beta)}\leq\frac{1-3\delta}{4(1-\delta\beta)}=\frac{1-2\delta-\beta+\delta\beta}{4(1-\delta\beta)}< C(R),
\end{align*}
and hence we can assume that $|t_i|\leq2$ if $a_i=R$.
Suppose $a_k=R+1$ and $t_{k}\geq 3$. Then, $d_k^-\geq 3\bar\alpha_k-2\bar\alpha_k\bar\alpha_{k-1}$, $d_k^+\geq -2\alpha_k$, and as $l\rightarrow \infty$
\begin{align*}
    s_3(k) & \leq\frac{(1-4\bar\alpha_k+2\bar\alpha_k\bar\alpha_{k-1})(1+\alpha_k)}{4(1-\bar\alpha_k\alpha_k)}\rightarrow\frac{(1-4\delta+2\delta\beta)(1+\beta)}{4(1-\delta\beta)}\\
    &=\frac{(1-4\delta+\beta-2\delta\beta+2\delta\beta^2)}{4(1-\delta\beta)}
    =\frac{(1-2\delta-\beta+2\delta\beta^2)}{4(1-\delta\beta)}< C(R).\qedhere
\end{align*}
\end{proof}
We now also assume that $|t_i|\leq2$ for all sufficiently large $i$.
\begin{lemma}\label{2} If $\gamma$ has infinitely many of the following blocks, then $ M(\alpha,\gamma)< C(R)$.
\begin{enumerate}[label=(\roman*)]
\item $t_k=1$ and $t_{k+1}=-2$.
        \item $t_k=0$ and $t_{k+1}=-2$.
\end{enumerate}
\end{lemma}
\begin{proof}$(i)$:
Plainly $d_k^-\leq\bar\alpha_k$, $d_k^+\leq-2\alpha_k+2\alpha_k\alpha_{k+1}$, and as $l\rightarrow \infty$
\begin{align*}
    s_1(k)\leq\frac{(1-3\alpha_k+2\alpha_k\alpha_{k+1})}{4(1-\bar\alpha_k\alpha_k)}\rightarrow \frac{(1-3\beta+2\beta^2)}{4(1-\delta\beta)}=\frac{(1-2\delta-\beta+2\delta\beta^2)}{4(1-\delta\beta)}< C(R).
\end{align*}
\noindent
$(ii)$: With $\theta=[0;\overline{R+1}]^-$ we have $\alpha_k\geq\cfrac{1}{R-\theta}=:\lambda$. Since $d_k^-\leq2\bar\alpha_k\bar\alpha_{k-1}\leq 2\beta \bar\alpha_k$, $d_k^+\leq-2\alpha_k+2\alpha_k\alpha_{k+1}$
and $\alpha_k\leq \beta$ 
\begin{align*}
    s_1(k)&\leq\frac{(1-(1-2\beta)\bar\alpha_k)(1-3\alpha_k+2\beta^2)}{4(1-\bar\alpha_k\beta)}\\ & \leq\frac{(1-\delta+2\beta \delta)(1-3\lambda+2\beta^2)}{4(1-\delta\beta)}
  \leq\frac{(1-3\lambda+2\beta^2)}{4(1-\delta\beta)} < C(R),
\end{align*}
using Lemma \ref{calculus} and $\bar\alpha_k>\delta$ for the second inequality. For the last inequality observe
that  $\beta-\lambda=\beta\lambda(\beta-\theta)<\lambda \beta^2$ so that
$1-3\lambda+2\beta^2< 1-3\beta +\beta^2(2+3\lambda)$ while  $(1-2\delta)(1-\beta)=1-3\beta+\beta^2(4-4\delta)$. 

\end{proof}
From Lemmas \ref{>2} and \ref{2} we see that a $\gamma$ with infinitely many $|t_i|\geq 2$ has $M(\alpha,\gamma)\leq C_0(R)<C(R)$ as $l\rightarrow \infty$ (where $C_0(R)$ is made explicit in the proof).
Hence $\lim_{l\rightarrow \infty}\rho(\alpha)=\lim_{l\rightarrow \infty} M(\alpha,\gamma^*)=C(R). $ \qed

\section{Proof of Theorem \ref{asymptotic}}
Suppose that $\alpha=[0;\overline{R,R,R+1,R,R+1,R+1,R,R+1,R+1,R,R+1}]^-$ with $R$ odd. By Theorem \ref{mymaintheorem} and \eqref{AsympOdd} we have
 \begin{align*}
     \rho(\alpha)\geq M(\alpha,\gamma^*)\geq C(R)=\frac{1}{4}\left(1-\frac{3}{R}+\frac{4}{R^2}+O(R^{-3})\right),
 \end{align*}
so we just need to show that all $\gamma $ have
\be \label{enough}    M(\alpha,\gamma)\leq \frac{1}{4}\left(1-\frac{3}{R}+\frac{4}{R^2}+O(R^{-3})\right). \ee

 We observe the following  
\begin{align*}
  [0;R,R \text{ or } R+1,\cdots]^-&=\frac{1}{R}+O(R^{-3}), \\
[0;R+1,R \text{ or } R+1,\cdots]^-&=\frac{1}{R}-\frac{1}{R^2}+O(R^{-3}),
\end{align*}
and so for $\alpha$ we have
\be \label{denom}
    \frac{1}{1-\bar\alpha_k\alpha_k}=1+\bar\alpha_k\alpha_k+(\bar\alpha_k\alpha_k)^2+\cdots=1+\frac{1}{R^2}+O(R^{-3}).
\ee

Now if $\gamma$ has $t_k=a_k$ infinitely often, then from $(\ref{s3s4})$ 
 \begin{align*}
     M(\alpha,\gamma)\leq\frac{1}{4}\left(\frac{1}{R}+O(R^{-3})\right)\left(1+\frac{1}{R^2}+O(R^{-3})\right)=\frac{1}{4}\left(\frac{1}{R}+O(R^{-3})\right),
 \end{align*}
so we can assume that $\gamma$ has only finitely many $t_i=a_i$. In view of \eqref{denom} we write
$$ \tilde{s}_j(k)=4(  1-\bar\alpha_k\alpha_k)s_j(k),\quad j=1,\ldots,4, $$
and \eqref{enough} amounts to showing that there are infinitely many $k$ with an
\be \label{2enough} \tilde{s}_j(k) \leq 1-\frac{3}{R}+\frac{3}{R^2} +O(R^{-3}).\ee
We proceed as in the proof of Theorem \ref{myexample} successively eliminating blocks of $t_i$, recalling that when we eliminate a block we also eliminate its negative or reverse (by interchanging $s_j(k)$).

By \eqref{1s3s4} we will get \eqref{2enough} if $\gamma$ has infinitely many $k$ with
\be \label{smallt} d_k^-\geq \frac{3}{R}-\frac{3}{R^2} +O(R^{-3}). \ee
We use this to rule out large $|t_i|$. If $t_k\geq 5$ then we have 
$$d_k^-\geq 5\bar\alpha_k+ \bar\alpha_k d_{k-1}^-> 4\bar\alpha_k\geq \frac{4}{R}+O(R^{-2}).$$
 If we have $t_k=4$ with $|t_{k-1}|\leq 4$ then $d_{k-1}^-=O(R^{-1})$ and 
again 
$$d_k^-=  \frac{4}{R}+O(R^{-2}).$$ 
So we can assume that $|t_i|\leq 3$ for all but finitely many $i$. If we have infinitely many blocks $a_k,a_{k-1}=R,R+1$ (or their reverse) with $t_k=3$ then 
$$d_k^-\geq 3\bar\alpha_k-2\bar\alpha_k\bar\alpha_{k-1}+O(R^{-3}) = \frac{3}{R}-\frac{2}{R^2}+O(R^{-3}).$$
Hence we can assume that (all but finitely many) $t_i=\pm 1$ if $a_i=R$ and $t_i=0,\pm 2$ if $a_i=R+1$.

First we rule out infinitely many consecutive positive or consecutive negative $t_i$. If $t_k,t_{k+1}>0$ then $d_k^-\geq \bar\alpha_k +O(R^{-2}),d_k^+\geq \alpha_k+O(R^{-2})$ and
$$ \tilde{s}_3(k)\leq (1-2\bar\alpha_k +O(R^{-2}))((1-2\alpha_k +O(R^{-2}))=1-\frac{4}{R}+O(R^{-2}).  $$

Next we rule out infinitely many blocks $t_{k-1},t_{k},t_{k+1},t_{k+2}=0,1,0,0$ (or their reverse $0,0,1,0$ or their negatives) since $d_{k}^-=\bar\alpha_k +O(R^{-3})$, $d_k^+=O(R^{-3})$ and
\begin{align*} \tilde{s}_3(k) &= \Big(1-2\bar\alpha_k+O(R^{-3})\Big)\Big(1-\alpha_k+O(R^{-3})\Big)\\
 & =\left(1-\frac{2}{R}+O(R^{-3})\right)\left(1-\frac{1}{R}+\frac{1}{R^2}+O(R^{-3})\right)= 1-\frac{3}{R}+\frac{3}{R^2}+O(R^{-3}). 
\end{align*}

If $t_k,t_{k+1}=2,0,$  then $d_k^-= 2\bar\alpha_k +O(R^{-2})$, $d_k^+=O(R^{-2})$ and
\begin{align*} \tilde{s}_3(k) & =\Big(1-3\bar\alpha_k+O(R^{-2})\Big)\Big(1-\alpha_k+O(R^{-2})\Big)\\
 & =\left(1-\frac{3}{R}+O(R^{-2})\right)\left(1-\frac{1}{R}+O(R^{-2})\right)= 1-\frac{4}{R}+O(R^{-2}). 
\end{align*}
Hence blocks $R+1,R+1$ must eventually have $t_k,t_{k+1}=0,0$ or $2,-2$ or $-2,2$.

If $t_{k-1},t_k,t_{k+1}=-2,1,-2,$  then $d_k^-= \bar\alpha_k -2\bar\alpha_k\bar\alpha_{k-1}+O(R^{-3})$, $d_k^+\leq -2\alpha_k+2\alpha_k\alpha_{k+1}+O(R^{-3})$ and
\begin{align*} \tilde{s}_1(k) & \leq \Big(1-2\bar\alpha_k\bar\alpha_{k-1}+O(R^{-3})\Big)\Big(1-3\alpha_k+2\alpha_k\alpha_{k+1}+O(R^{-3})\Big)\\
 & =\left(1-\frac{2}{R^2}+O(R^{-3})\right)\left(1-\frac{3}{R}+\frac{5}{R^2}+O(R^{-2})\right)= 1-\frac{3}{R}+\frac{3}{R^2}+O(R^{-3}). 
\end{align*}
Hence if we have a block $a_{k-2},a_{k-1},a_{k},a_{k+1},a_{k+2}=R+1,R+1,R,R+1,R+1$ with $t_{k}=\pm 1$ then we must have $t_{k+1},t_{k+2}=0,0$ and $t_{k-1},t_{k-2}=\mp 2,\pm 2$ (or vice versa in which case we use the reverse). Consider then the block
$$a_{k+1},\ldots ,a_{k+6}=R+1,R+1,R,R+1,R,R,\quad t_{k+1},t_{k+2}=0,0.$$
Assuming that  $t_{k+3}=1$ (or use the negative), then having ruled out $0,0,1,0$, we must have $t_{k+4},t_{k+5},t_{k+6}=-2,1,-1$ and finally
\begin{align*} \tilde{s}_{1}(k+4) & =\Big( 1-3\bar\alpha_{k+4}+\bar\alpha_{k+4}\bar\alpha_{k+3}+ O(R^{-3}\Big)\Big( 1-\alpha_{k+4}\alpha_{k+5}+O(R^{-3})\Big) \\
=& \left(1-\frac{3}{R}+\frac{4}{R^2}+O(R^{-3})\right)\left( 1-
\frac{1}{R^2}+O(R^{-3}\right) =  1-\frac{3}{R}+\frac{3}{R^2}+O(R^{-3}).  \qed
\end{align*}
 
\section{Proof of \eqref{otheralpha} }

 If in the previous proof we had taken $\alpha$ to have period $R,(R+1,)^l$, $R\geq 5$ odd,   then \eqref{enough} would still hold, except for those $\gamma$ whose $t_i$ eventually consist of zeros one side of the $\pm 1$ and blocks of $\mp 2,\pm 2$ the other. For these $\gamma$, 
if $t_k=1$ inside a block $\ldots,0,0,1,-2,2,\ldots,$ then $d_{k}^-\rightarrow \delta$, $d_{k}^+\rightarrow -2\beta/(1+\beta)$ and $d_{k-1}^-\rightarrow 0$, $d_{k-1}^+\rightarrow \delta-2\delta \beta/(1+\beta)$ as $l\rightarrow \infty$ and
$$ s_3(k-1), s_1(k) \rightarrow \frac{ 1-3\beta +\frac{2\beta^2}{1+\beta}}{4(1-\delta\beta)}=\frac{\left(1-2\delta +\frac{2\delta \beta}{1+\beta}\right)(1-\beta)}{4(1-\delta\beta)}=:C_1(R), $$
with  $s_1(k-1)>s_3(k-1)$ and $s_3(k)> (1-2\delta)/4(1-\alpha\beta) >C_1(R)$.  Likewise for the negatives and reverses.
At all places 
$$s_2(k),s_4(k)>(1-\bar\alpha_k)(1-\alpha_k)/4(1-\beta\alpha_k)>(1-\delta)^2/4(1-\delta\beta)>C_1(R). $$
If $t_k,t_{k+1}=0,0$ and $d_k^-\rightarrow 0$, $\bar\alpha_k\rightarrow \beta$ (likewise if $d_k^+\rightarrow 0$, $\alpha_k\rightarrow \beta$), then 
$$s_1(k),s_3(k)\geq \frac{(1-\beta) (1-\alpha_k(1+\delta))}{4(1-\beta \alpha_k)}>\frac{(1-\beta)(1-\delta-\delta^2)}{4(1-\delta\beta)}>C_1(R).$$
If $t_k,t_{k+1}=-2,2$, then $d_{k-1}>\beta$ (either $t_{k-1}=2$ or $t_{k-1}=1$ with $d_{k-1}\rightarrow \delta$) 
and 
$$ s_1(k)\geq \frac{(1-\bar\alpha_k (3-\beta) )\left(1+\beta (1-2\delta)\right)}{4(1-\bar\alpha_k \beta)}\geq \frac{(1-3\delta+\delta \beta)( 1+\beta -2 \beta \delta)}{4(1-\delta\beta)}>C_1(R),$$
for $R\geq 9$ using
$$ (1-3\delta+\delta\beta)(1+\beta-2\beta\delta)=(1-3\beta+2\beta^2) + \beta( 1-9\beta +4\delta^2\beta),$$
replacing  $\bar\alpha_k$ by $1/(R+1-\delta)$ instead of $\delta$ in the second inequality and checking numerically for  $R=5$ and 7.
Likewise for $s_3(k)$ and for $2,-2$. Hence these $\gamma$ have $M(\alpha,\gamma)\rightarrow C_1(R)$ as $l\rightarrow \infty$.
The proof of Theorem \ref{asymptotic} immediately gives \eqref{otheralpha} for suitably large $R$. To see that it is true for all $R\geq 5$ we show that $M(\alpha,\gamma)< C_1(R)$ for the other $\gamma$.  Notice, if we let  $\l\rightarrow \infty,$ then$(1-\bar\alpha_k\alpha_k) ^{-1}\leq (1-\delta\beta)^{-1}$; so it will be enough to show that the remaining $\gamma$ have infinitely many
$$ \tilde{s}_j(k)\leq 1-3\beta +\frac{2\beta^2}{1+\beta} =\left(1-2\delta+\frac{2\delta\beta}{1+\beta}\right)(1-\beta)=:\tilde{C}_1(R). $$
We repeat the steps of the proof of Theorem \ref{asymptotic}; successive ruling out certain  blocks of $t_i$ (or their negatives and reverses) occurring infinitely often. We can rule out $t_k=a_k$ since
$ \delta < 1-3\beta <\tilde{C}_1(R).$ To eliminate large $|t_k|$ we replace \eqref{smallt} by
$$ d_k^- \geq 3\beta -\frac{2\beta^2}{1+\beta}, $$
successively ruling out infinitely many $t_k\geq 4$ using $d_k^->3\beta,$  then  $t_k=3$  using
$$d_k^-   \geq  3\delta - \frac{ 2\delta\beta}{1-\beta}>3\beta. $$ 

Now if $t_k=2$ and $t_{k+j}>0$ for some $j\geq 1$, with  $t_i=0$ for any $k<i<k+j$, then $d_k^-=2\bar\alpha_k+\bar\alpha_k d_{k-1}^-\geq \bar\alpha_k+\bar\alpha_k\bar\alpha_{k-1}, d_k^+\geq 0,$  and 
$$\tilde{s}_3(k)\leq(1-2\beta-\beta^2)(1-\beta)=(1-2\delta+2\delta\beta-\beta^2)(1-\beta)<\tilde{C}_1(R).$$ 
If $t_k,t_{k+1}=2,0,$ then
$$ \tilde{s}_3(k) \leq \left(  1-3\beta +2\beta \delta \right) \left( 1-\beta +2\beta\delta\right) =1-3\beta+2\beta^2-2\beta^3-\lambda < \tilde{C}_1(R), $$
with $\lambda=\beta(1-5\beta +2\beta^2(1-2\delta^2))>0.$

If $t_{k-1},t_k, t_{k+1}=0,1,0,$  then as $l\rightarrow\infty$
$$  \tilde{s}_3(k) \rightarrow ( 1-2\delta )( 1-\beta ) < \tilde{C}_1(R).  $$
Finally, if $t_{k-1},t_k,t_{k+1}=-2,1,-2,$ then
$$ \tilde{s}_1(k) \leq \left(  1- \frac{2\delta\beta}{1+\beta}   \right) \left(1-3\beta +\frac{2\beta^2}{1+\beta}\right)< \tilde{C}_1(R). \qed$$

\end{document}